\documentclass[12pt]{a_t} 

\usepackage{graphicx}
\usepackage{subfigure}

\year{2024}

\authors{
N.M.~Markovich, DrHab\\
V.A. Trapeznikov Institute of Control Sciences of Russian Academy of Sciences, Moscow, Russia,\\
M.~Vai\v{c}iulis, PhD\\
Vilnius University, Vilnius, Lithuania}

\begin{document}

\title{INVESTIGATION OF TRIANGLE COUNTS IN GRAPHS EVOLVED BY UNIFORM CLUSTERING ATTACHMENT}%


\maketitle

\begin{abstract}
The clustering attachment model introduced in the paper Bagrow and Brockmann (2013) may be used as an evolution tool of random networks. We propose a new clustering attachment model which can be considered as the limit of the former clustering attachment model as model parameter $\alpha$ tends to zero. We focus on the study of a total triangle count that is considered in the literature as an important characteristic of the network clustering. It is proved that total triangle count tends to infinity a.s. for the proposed model. 
Our 
simulation study is used for the modeling of sequences of triangle counts. It is based on the interpretation of the clustering attachment as a generalized  P\'{o}lya-Eggenberger urn model that is introduced here at first time.
\\
	\textit{Keywords:} clustering attachment, preferential attachment, random graph, evolution, urn model.
\end{abstract}

\section{Introduction} \label{sec0}

\par
In this paper we discuss an evolution model of undirected random graphs by
the clustering attachment (CA) rule introduced in \cite{Bag}. The attachment of a newly appending node to 
existing nodes by the CA typically lowers its clustering coefficient, in contrast to the well-known preferential attachment (PA) that executes so called 'rich-get-richer' models. This leads to numerous new phenomena like a community formation around new nodes, bursts of a modularity and light-tailed distributions of node degrees \cite{Bag}. The idea of the CA is that existing nodes are drawn not towards hubs,
but towards densely connected groups that is usual for social human behaviour.
\par
We propose a new clustering attachment model which can be considered as the limit of the former clustering attachment model in \cite{Bag} as model parameter $\alpha$ tends to zero.
The latter model  This model boils down to the sequential selection of existing nodes
by a new node with equal probability. The following example of such an attachment can be given. An individual comes to an unfamiliar community (for a new job) and chooses someone to be friends at random, with equal probability. This can happen to every beginner. If a newcomer joined an existing node, this does not mean that another new node will join it in the future. On the contrary, the chances of joining decrease as there are more and more nodes to choose from.
\par
A triangle of connected nodes is the most studied subgraph that can be considered as a basic community. It seems, it was introduced first in \cite{New2009}. 
The triangle count  relates  to calculation of the clustering coefficient. The latter  is an important  representation of a clustering structure in random graphs since it expresses the fraction of connected neighbors of a node. The limiting behavior of the number of triangles attracts the attention of many researchers, see \cite{LiuDong}, \cite{Bobkov}, \cite{Gar} among others. Our objective is to find a limit behaviour of the total triangle counts in the evolving CA graph.
Remarkable is that the triangle counts were investigated for the PA evolution model. Here, we deal with the CA setting that is a novelty.
 \par
Let $G_n=\left( V_n, E_n\right)$,
$n=1,2, \dots$ denote the sequence of random graphs in which $n$ may
be interpreted as the discrete time since graphs are generated
during the evolution. Here, $V_n$ is the set of nodes and $E_n$ is
the set of edges. Here and further, we apply notations used in \cite{Hof}.  Let $\# A$ denote the cardinality (the number of
elements) of arbitrary finite set $A$. The graph $G_n$ consists of
$\#V_1+n-1$ nodes and $\#E_1+m(n-1)$ edges, i.e., for any $n \in
\mathbb{N}$,
$$
\#V_n =\#V_1+n-1, \quad \#E_n=\#E_1+m(n-1),
$$
where natural $m \ge 2$ is a parameter of the  clustering attachment
model.
\par
Before providing postulates that describe how new edges and nodes
are formed, let us recall that the clustering coefficient of node $i \in
V_n$ is defined as
\begin{eqnarray*}
c_{i,n}&=& \left\{
  \begin{array}{ll}
    0, & \hbox{$D_{i,n}=0$ or $D_{i,n}=1$,} \\
    2 \Delta_{i,n}/\left(D_{i,n}\left(D_{i,n}-1\right) \right), & \hbox{$D_{i,n} \ge 2$,}
  \end{array}
\right.
\end{eqnarray*}
where $D_{i,n}$ and $\Delta_{i,n}$ are the degree and  number of
triangles of node $i$, respectively. Let us define more general definition of the CA model.
\\
\begin{definition}
The graph $G_{n+1}$ is obtained
from $G_n$  according to the rule, which consists of two parts:

(i) the deterministic part: a new node $\#V_1+n \in V_{n+1}
\setminus V_n$ is appended to $G_n$;

(ii) the probabilistic part: each node $i \in V_n$ is equipped with
the weight
\begin{equation}\label{a02}
p_{i,n}=\frac{f\left(c_{i,n} \right)+\epsilon}{\sum_{j \in V_n}
\left( f\left(c_{j,n} \right)+\epsilon \right)},
\end{equation}
where the  attachment function $f:[0,1] \to [0, \infty)$ is
deterministic, non-decreasing, while $\epsilon \ge 0$  is another
parameter of the CA model. Choose the existing
nodes $i_1, \dots i_m \in V_n$ by a successive sampling and connect
each of them with the newly appended node $\#V_1+n$.
\end{definition}
By replacing $c_{i,n}$ by $D_{i,n}$ in (\ref{a02}), we turn
to the definition of the PA model, see, e.g., pg. 5
in \cite{Wan}. For the CA model
introduced in \cite{Bag}, 
weights (\ref{a02}) 
were constructed by the
attachment function
\begin{equation}\label{a02-0}
    f(x)=x^{\alpha}, \ \alpha>0.
\end{equation}
Further we will assume, in addition, that $f(x)$ 
in (\ref{a02}) satisfies conditions
\begin{equation}\label{a02-1}
    f(0)=0, \quad f(x)>0, \ 0<x\le 1.
\end{equation}
Assumptions regarding the initial graph $G_1$ depend on the parameter
$\epsilon$. If $\epsilon>0$ holds, then it is enough to assume that
\begin{equation}\label{a03}
    \#V_1 \ge m.
\end{equation}
If $\epsilon=0$, then 
the assumption 
stronger than (\ref{a03}) is required.
Namely, we assume that
\begin{equation}\label{a04}
    \#\tilde{V}_1 \ge m,
\end{equation}
where the sequence of sets $\tilde{V}_1, \tilde{V}_2, \dots$ is
defined as follows: $\tilde{V}_n=\{ i \in V_n: \
f\left(c_{i,n}\right)>0\}$, $n \ge 1$.
Under assumption (\ref{a02-1}) the
sequence $\tilde{V}_1, \tilde{V}_2, \dots$  is non-decreasing in the sense that ${\tilde
V}_n \subseteq
{\tilde V}_{n+1}$ for each $n \in \mathbb{N}$. Thus,
from (\ref{a04}) it follows $\#\tilde{V}_n \ge m$ for each $n \in
\mathbb{N}$. The last inequality ensures that for any $n \in
\mathbb{N}$ there is at least one collection of nodes $i_1, \dots,
i_m \in V_n$ that can be
chosen by 
the  successive sampling with a positive probability.

The  CA model has a Markov property in the sense
that the 
obtaining $G_{n + 1}$ once
$G_n$ is known does not require a prior knowledge of earlier graphs
$G_1, G_2, \dots, G_{n-1}$.



The paper is organized as follows. The main result is given in section  \ref{sec1.1}. Section  \ref{sec1.2} contains a survey of known results.
In Section \ref{sec2} we discuss
the  successive sampling for the CA model
considered in Theorem \ref{main}. The proves of Theorem \ref{main}, corrolary \ref{cor} and lemms required for the proof of the theorem are
given in Section \ref{sec3}. Section \ref{sec4} contains
results of a 
simulation study to demonstrate a growth of the total
triangle counts.  Section \ref{sec7} contains a gratitude. Conclusions are presented in  Section \ref{sec5}.

\section{Main result}\label{sec1.1}

The number of triangles in the graph $G_n$ is expressed in terms of the number of triangles of nodes
$$
\Delta_n=\frac{1}{3}\sum_{i \in {\tilde V}_n} \Delta_{i,n}.
$$
Let us consider the CA model with parameters
$\epsilon=0$, $m=2$ and the attachment
function
\begin{equation}\label{a05}
f(x)=\left\{
       \begin{array}{ll}
         0, & \hbox{$x=0$,} \\
         1, & \hbox{$x>0$.}
       \end{array}
     \right.
\end{equation}
A sequence of functions $f_n=x^{1/n}$, $n=1,2,\dots$ converges
to $f(x)$ in (\ref{a05}) point-wise. This allows us to call
such CA model
as a limit model (as
$\alpha \downarrow 0$) of the  CA model with
parameters $\epsilon=0$, $m=2$ and the attachment function (\ref{a02-0}).

\begin{theorem} \label{main}
Let  $G_n$, $n=1,2, \dots$ be a sequence of  graphs generated by the  CA model with
parameters $\epsilon=0$, $m=2$ and the attachment function given in
(\ref{a05}).
Assume that the initial graph $(V_1, E_1)$ is  finite and
satisfies condition (\ref{a04}). Then, as $n \to \infty$,
\begin{equation}\label{a06}
    \Delta_n \to \infty,\qquad n\to\infty\qquad \ {\mbox a.s.}
\end{equation}
\end{theorem}
As it will be shown in Section \ref{sec2}, the nodes of the sequence ${\tilde V}_1, {\tilde V}_2, \dots$ play a key role in the evolution of the CA model.
\begin{corollary} \label{cor0}
    Under the assumptions of Theorem \ref{main}, we get
$\# {\tilde V}_n \to \infty$  as $n \to\infty$ a.s..
\end{corollary}
%
Let the sequence of the sets
${\tilde E}_1, {\tilde E}_2, \dots$ be such that for any $n\ge 1$,
$\left({\tilde V}_n, {\tilde E}_n \right)$
is a complete graph. In other words, ${\tilde E}_n$ denotes the set of all possible edges between nodes in ${\tilde V}_n$. Some edges in ${\tilde E}_n$ may not exist in the graph $G_n$.
The next corollary contains a lower bound of the expectation $E \Delta_n$.
\begin{corollary} \label{cor}
    Under the assumptions of Theorem \ref{main}, it holds
    \begin{equation}\label{a07}
E \Delta_n -\Delta_1>
\frac{\#\left({\tilde E}_1 \cap E_1 \right)}{ 3 \# {\tilde E}_1} \ln(n-1),
\quad n \ge 2.
\end{equation}
\end{corollary}

\section{Related works}\label{sec1.2}
Let us firstly review the literature dealing with the CA model.
The CA rule was studied in \cite{Bag} as a
pioneer paper. The term 'clustering' with regard to social networks
was used in \cite{New} for the  time evolution in scientific
collaboration networks where the average clustering coefficient is high. In
contrast to the PA which is intensively studied
in the literature due to its application in real-world networks like Internet, the
study of the CA seems to be did not get noticeable
development.

In \cite{Bag}, the CA model with parameters  $\epsilon\ge 0$, $m=2$
and  the attachment function (\ref{a02-0}) is considered  using heuristics and Monte-Carlo simulations.
The main conclusions of \cite{Bag} are the following: (i)
the model parameters $\epsilon$, $m$ and  $\alpha$  impact
on bursts  of the modularity time series built during evolution
with a significant role of $\alpha$ (see, e.g., \cite{New2}, \cite{Arenas}
for definitions of the modularity); (ii) fluctuations of the modularity for
$\alpha >0$ are far larger than observed for $\alpha= 0$; (iii)
the consecutive growth and decay of the modularity values reflect
to the arising of dense communities or visa versa, and in general,
the evolution of the communities; (iv) the average clustering
coefficient increases significantly as $\alpha$
increases.  In \cite{Mar}, the impact of
$\alpha$ on the modularity sequence is investigated for the same model as in  \cite{Bag}, but with an uniform node deletion at each evolution step.
Based on Monte-Carlo simulations,  a weak clustering of the modularity and a weak heaviness of node degree
distribution tails were observed in \cite{Mar} as $\alpha>0$.
Strong consecutive  bursts of the modularity
and
light tail of  the node degree distribution were recognised in \cite{Mar} for the case $\alpha=0$.

Our paper is devoted  to the limit theorem for the total triangle count in the graph. Theorem \ref{main} is the first theoretical result for
the CA model. Let us recall the corresponding results for the PA model and several other models.
The PA model was first introduced in \cite{Bar1}, see also \cite{Bar2}. In \cite{Bol}  the expected number of triangles is proved to be of rate 
$\ln^3(n)$ for the original Albert-Barab\'{a}si model (that is the PA model with $\epsilon=0$, $m=2$ and the attachment function $f(x)=x$).
The PA model with $\epsilon>0$, $m \ge 2$ and $f(x)=x$ was studied in \cite{Ege}.
In \cite{Ege} it is observed that
the expected number of triangles is of order $\ln(n)$. In \cite{Pro} (see also references therein) and \cite{Gar} results regarding the number of
triangles for a broad class of PA models  can be found.

Let $\tau$ denote the exponent of a power-law distribution which, as a rule,  approximates well the node degree distribution in real-world networks. The  relation
\begin{equation} \label{a08}
\frac{\Delta_n}{n^{3(3-\tau) / 2}}  \ {\buildrel \rm p \over \to} \ C, \quad n \to \infty
\end{equation}
holds for several random graph models. Here, ${\buildrel \rm p \over \to}$ denotes the convergence in probability
and the constant $C>0$ depends on parameters of a concrete model.
In particular, in \cite{Ste} it was shown that (\ref{a08}) holds for the inhomogeneous random graphs.
If $\tau< 7/3$, the total number of triangles satisfies (\ref{a08}) for a geometric inhomogeneous random graph, while the
normalization in (\ref{a08}) must be replaced by $n$ if $\tau> 7/3$ holds, see \cite{Mic}. In \cite{Gao}, it is obtained that
(\ref{a08}) holds for uniform random graphs and the erased configuration model.

The convergence with rate $O\left(n^{-1 / 2}\right)$ in total
variation between the distribution of triangle counts and a Poisson
distribution in generalized random graphs with random vertex weights introduced in \cite{Bri}
is proved  in \cite{Bobkov}.

%

\section{Consecutive uniform sampling of nodes} \label{sec2}

Let $G_n$ be an observed graph. Our objective here is to show that the sets $\tilde{E}_n$ and $\tilde{E}_n \cap E_n$ are enough to know for the evolution by the CA.
By combining 
(\ref{a02}) and (\ref{a05}) we get the equivalent definition of ${\tilde V}_n$:
$\tilde{V}_n=\{ i \in V_n: \ \Delta_{i,n}>0\}$.
By substituting
$\epsilon=0$, $m=2$ and (\ref{a05}) into (\ref{a02}) we obtain that
attachment weights $p_{i,n}$, $i \in V_n$ are uniformly distributed at $\tilde{V}_n$:
\begin{equation} \label{b01}
p_{i,n}=
\left\{
  \begin{array}{ll}
    0, & \hbox{$i \in V_n\setminus {\tilde V}_n$,} \\
    1/\#{\tilde V}_n, & \hbox{$i \in {\tilde V}_n$.}
  \end{array}
\right.
\end{equation}
By (\ref{b01})  it follows that an appended node $\# V_1+n$ can be attached with positive probability only   to existing nodes $\{i \in V_n$: $\Delta_{i,n}>0\}$ involved into at least one triangle.

Let the sequence of sets $\mathcal{E}_1, \mathcal{E}_2, \dots$ be such that for any $n \ge 1$, the pair $\left(V_n, \mathcal{E}_n \right)$
is a complete graph.



Let $W_n$ denote a pair of nodes $\{i_1, i_2\}$ from $V_n$, chosen by applying the successive sampling (or random sampling without return). The  successive sampling begins by selecting node $i_1$ from $V_n$ with probability  (\ref{b01}).
 The probabilities of remaining $\#V_n-1$ nodes
are then re-normalized as follows: $p_{i,n}/(1-p_{i_1,n})$. The process is repeated by
selecting the node $i_2$. 
It does not matter for the evolution by the CA model
which node is chosen first in the pair.
Thus, further we will assume that $W_n$ is a two-dimensional random vector with values in $\mathcal{E}_n$.

\begin{proposition} \label{prop1}
Random vector $W_n$ is uniformly
distributed on ${\tilde E}_n$.
\end{proposition}
\begin{proof}
Let $\{i, j\} \in {\tilde E}_n$. Let us consider for each node $i \in {\tilde V}_n$ the
random event
$A_{n,i}
\equiv\{\mbox{node} ~i~\mbox{is ~ chosen} \}$.
By the law of a total probability,
\begin{eqnarray}
P\left(W_n=\{i,j\}\right) &=& P(A_{n,j}|A_{n,i})P(A_{n,i})+P(A_{n,i}|A_{n,j})P(A_{n,j})
\nonumber \\
&=&
\frac{2}{ \#{\tilde V}_n \left(\#{\tilde V}_n-1\right)}
=\frac{1}{\# {\tilde E}_n}.
\label{20}
\end{eqnarray}
By  (\ref{20}), we obtain
$$
P\left(W_n \in {\tilde E}_n\right)=\sum_{\{i_1, i_2\} \in {\tilde E}_n} P\left(W_n=\{i_1,i_2\}\right)=1.
$$
This
implies $P\left(W_n \in \mathcal{E}_n \setminus{\tilde
E}_n\right)=0$, and consequently,
$P\left(W_n=\{i,j\}\right)=0$ for any pair of nodes $\{i,j\} \in \mathcal{E}_n \setminus{\tilde E}_n$.
\end{proof}

\bigskip

From Prop. \ref{prop1} it follows that the attachment rule (ii) with function (\ref{a05})  and parameters $\epsilon=0$, $m=2$ does not
give preferences to a pair of nodes  with relatively large clustering
coefficients. This means an absence of {\it the rich-get-richer} effect
in  the CA model under consideration.

The equality $P\left(W_n=\{i,j\}\right)=0$, $i, j \in V_n \setminus {\tilde V}_n$ means that
nodes from the set $V_n \setminus {\tilde V}_n$ do not take part in forming  the graph $G_{n+1}$ since newly appending nodes do not attach to such nodes. The possibility to decompose
the set $V_n$ into subsets of active nodes ${\tilde V}_n$ and inactive ones $V_n \setminus {\tilde V}_n$  turns to the mind that
the CA model can be used to generate various two-classes
communities. We refer \cite{Bag} for more possible applications of CA model.

\begin{example}\label{Exam2.1}
    Fig. \ref{fig1} illustrates  classes mentioned above of one realization of the CA  random graph $(V_n, E_n)$ with $n=5000$ and  $n=10000$ when the initial graph $G_1$ is a quadrilateral with all diagonals except one:
$$
V_1=\{1,2,3,4\}, \quad E_1=\{\{1,2\}, \{1,3\}, \{1,4\}, \{2,3\}, \{3,4\}\}.
$$
\end{example}
 \begin{figure}[ht!]
\begin{center}
 \includegraphics[width=0.40\textwidth]{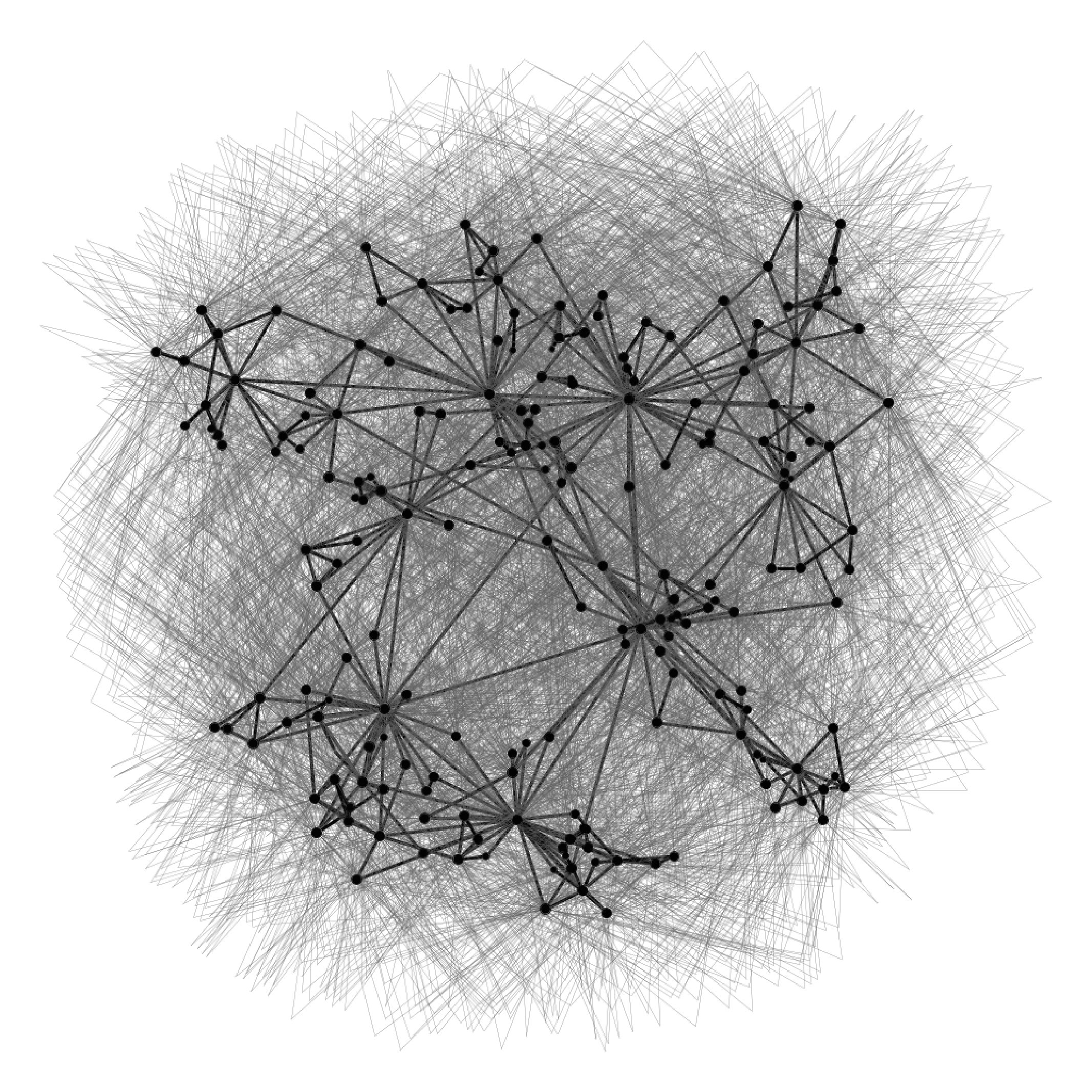} \quad \includegraphics[width=0.40\textwidth]{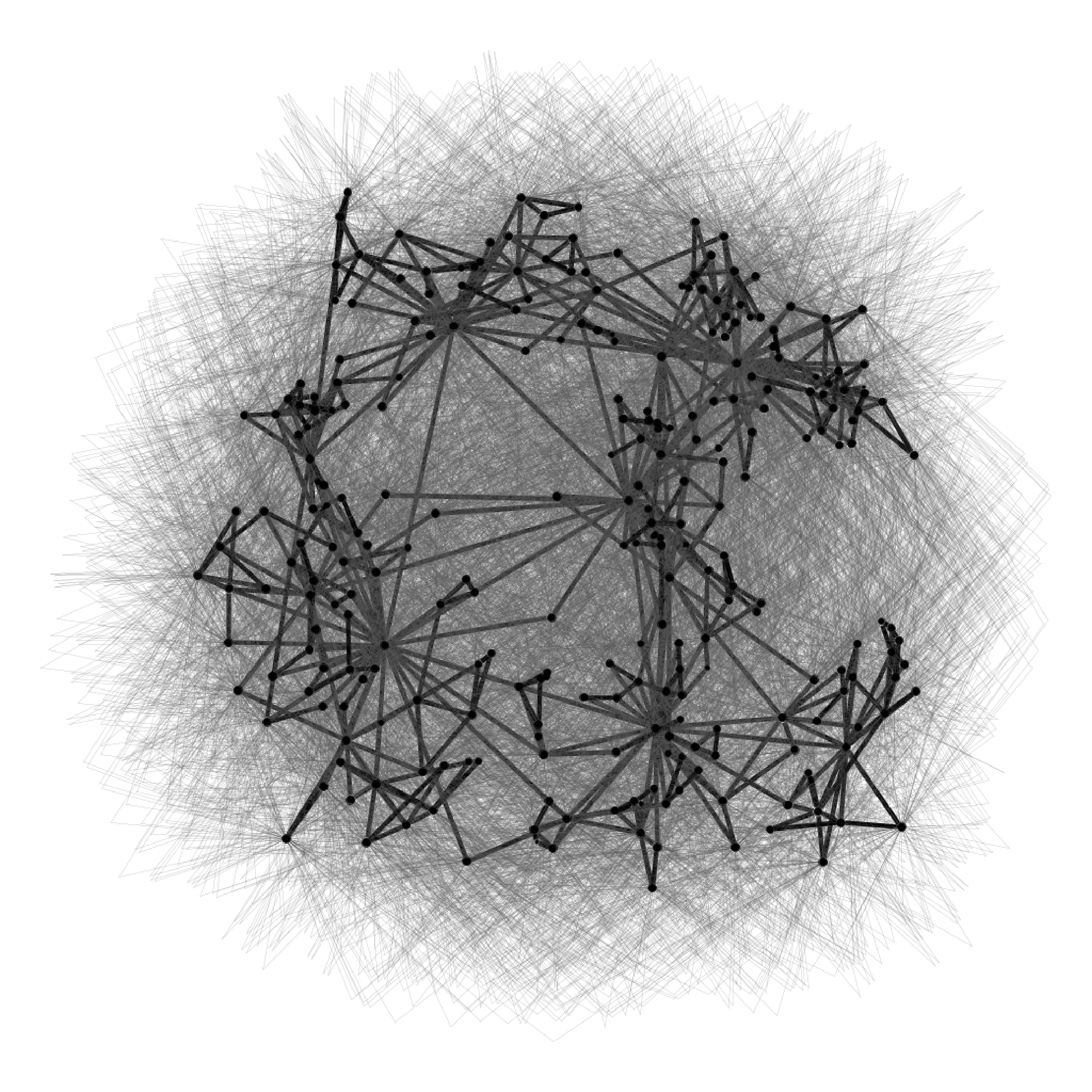}
\end{center}
\caption{The CA random graph  $(V_n, E_n)$ with $n=5000$ (left) and $n=10000$ (right): nodes from ${\tilde V}_n$ are shown in black, edges from ${\tilde E}_n \cap E_n$ are in dark grey and  the rest of graph  is in light grey.} \label{fig1}
\end{figure}

\section{Proofs} \label{sec3}

\subsection{Auxiliary lemmas}
\par
Here and throughout, we assume that set of edges $E_n$, $n \ge 1$ is presented as a set of unordered pairs of nodes.
For any $n \ge 1$ we divide the set ${\tilde E}_n$ into two subsets: ${\tilde E}_n \cap E_n$ and
${\tilde E}_n \setminus \left({\tilde E}_n \cap E_n\right)$. The set ${\tilde E}_n \cap E_n$ contains the unordered pairs of nodes which are connected
by edges
in graph $G_n$.
Put $B_n=\{W_n \in {\tilde E}_n \cap E_n\}$, $n \ge 1$ denotes a random event that a pair of nodes $W_n$ chosen by the uniform sampling  is connected by an edge from $E_n$.

Let $B^c$ denote the opposite event to the random event $B$.
Let $n,k  \in \mathbb{N}$,
and $\ell \in \mathbb{N}\cup \{0\}$ are such that $n\ge 2$,
$1 \le k<n$, $0\le \ell \le n-k$. Let $\left(B_k \cap \dots \cap B_{n-1}\right)_{\ell}$ denote the union of $(n-k)!/(\ell! (n-k-\ell)!)$ non-intersecting random events such that 
$\ell$ events $B_j$ are replaced by
their opposites in the intersection $B_k \cap \dots \cap B_{n-1}$.
If a random event $\left(B_k \cap \dots \cap B_{n-1}\right)_{\ell}$ occurs, then
$\ell$ pairs of nodes among $W_k, \dots, W_{n-1}$ are disconnected, while the rest $n-k-\ell$ pairs are connected.
For example,
if $k=1$, $n=4$, then we have
\begin{eqnarray*}
\ell=0: \quad (B_1\cap B_2 \cap B_3)_{0} &=& B_1\cap B_2 \cap B_3, \\
\ell=1: \quad (B_1\cap B_2 \cap B_3)_{1} &=&(B_1\cap B_2 \cap B_3^c) \cup (B_1\cap B_2^c \cap B_3) \cup (B_1^c\cap B_2 \cap B_3), \\
\ell=2: \quad (B_1\cap B_2 \cap B_3)_{2} &=&(B_1\cap B_2^c \cap B_3^c) \cup (B_1^c\cap B_2 \cap B_3^c) \cup (B_1^c\cap B_2^c \cap B_3), \\
\ell=3: \quad (B_1\cap B_2 \cap B_3)_{3} &=&
    B_1^c\cap B_2^c \cap B_3^c.
\end{eqnarray*}
By Prop. \ref{prop1},
\begin{equation} \label{c00}
P\left(B_n \right)=\sum_{\{i_1, i_2\} \in {\tilde E}_n \cap E_n} \frac{1}{\# {\tilde E}_n}=\frac{\# \left( {\tilde E}_n \cap E_n\right)}{\# {\tilde E}_n}.
\end{equation}
Let $n \ge k-1$ and $k \ge 1$. If the random event
$\left(B_k \cap \dots \cap B_{n-1}\right)_{\ell}$ has occurred, then
\begin{eqnarray*}
  \#\left({\tilde E}_n \cap E_n \right)&=& \# \left( {\tilde E}_k \cap E_k\right)+2(n-\ell-k), \\
  \#{\tilde E}_n &=& \# {\tilde E}_k +(n-\ell-k)\# {\tilde V}_k+(1/2) (n-\ell-k)(n-\ell-k-1).
\end{eqnarray*}
Thus, by (\ref{c00}) we get
\begin{eqnarray*}
P\left(B_n| \left(B_k\cap \dots \cap B_{n-1}\right)_{\ell}\right)&=& p_{n-\ell}\left( G_k\right),
\end{eqnarray*}
where
\begin{equation}\label{c03-n1}
p_n\left( G_k\right)=\frac{\# \left( {\tilde E}_k \cap E_k\right)+2(n-k)}{\# {\tilde E}_k +(n-k)\# {\tilde V}_k+(1/2) (n-k)(n-k-1)}.
\end{equation}

Let us prove several inequalities which will be used in the proof of Theorem \ref{main}.
\begin{lemma} \label{lem0} Let assumptions of Theorem \ref{main} hold.
For any natural numbers  $k \le n$,
\begin{eqnarray} \label{c03-0a}
p_n\left( G_k\right)&>&0, \\
\label{c03-0}
p_n\left( G_k\right) &>& p_{n+1}\left( G_k\right).
\end{eqnarray}
\end{lemma}
\begin{proof} Assumption (\ref{a04}) implies ${\tilde E}_k \not = \emptyset$  and ${\tilde E}_k \cap E_k \not = \emptyset$ for any natural $k$. Now (\ref{c03-0a}) follows from (\ref{c03-n1}).

Next, using (\ref{c03-n1}), the inequality
(\ref{c03-0}) boils down to the following
$$
\#{\tilde V}_k\left\{\#\left( {\tilde E}_k \cap E_k\ \right)- \#{\tilde V}_k+1 \right\}+(n-k)\#\left( {\tilde E}_k \cap E_k\ \right)+(n-k)(n-k+1)>0.
$$
Thus, it is enough to show that for any initial graph satisfying (\ref{a04}),
\begin{equation}\label{c03-1}
    \#\left( {\tilde E}_k \cap E_k\ \right) >  \#{\tilde V}_k-1
\end{equation}
holds. Let us use the mathematical induction to prove (\ref{c03-1}). We consider the graph $\left({\tilde V}_1, {\tilde E}_1 \cap E_1\right)$.
A minimum number of edges in ${\tilde E}_1 \cap E_1$ depends on  $\# {\tilde V}_1$. If $\# {\tilde V}_1=3$ holds, then the graph, which has a
minimum number of edges is triangle for which  $\#\left( {\tilde E}_1 \cap E_1\ \right)=3$.
Thus, the inequality (\ref{c03-1})  holds for $k=1$ in this case.
If $\# {\tilde V}_1>3$ holds, then
the graph, which has a minimum number of edges consists of $\# {\tilde V}_1-2$ triangles with the same one side (see, Example \ref{Exam2.1} for the case $\# {\tilde V}_1=4$).
For such graphs it holds $\#\left({\tilde E}_1 \cap E_1\right)=\# {\tilde V}_1+1$. This yields that inequality
(\ref{c03-1}) with $k=1$ and $\# {\tilde V}_1>3$ holds again.

Next, let us make the induction hypothesis that (\ref{c03-1}) holds. We have to show that (\ref{c03-1})
holds  when $k$ is replaced by $k + 1$. If the random event $B_k$ occurs, then
$\#\left( {\tilde E}_{k+1} \cap E_{k+1}\ \right)=\#\left( {\tilde E}_k \cap E_k\ \right)+2$ and
$\#{\tilde V}_{k+1}=\#{\tilde V}_k+1$ follow. By  (\ref{c03-1}) we get
$$
\#\left( {\tilde E}_{k+1} \cap E_{k+1}\ \right)-\#{\tilde V}_{k+1}+1=\left\{\#\left( {\tilde E}_k \cap E_k\ \right)-  \#{\tilde V}_k+1 \right\}+1 > 0.
$$
The case when $B_k$ does not occur can be considered similarly.
\end{proof}



\begin{lemma} \label{lem2} Under assumptions of Theorem \ref{main}, we have
\begin{equation}\label{c06}
    P\left(B_n \right) > P\left(B_n| B_1\cap \dots \cap B_{n-1}\right), \quad n=2,3,\dots
\end{equation}
\end{lemma}
\begin{proof} By the law of total probability,
\begin{eqnarray*}
\!\!\!\!\!P(B_n)&=&\sum_{\ell=0}^{n-1} P\left(B_n|\left(B_1\cap \dots \cap B_{n-1}\right)_{\ell}\right)P\left(\left(B_1\cap \dots \cap B_{n-1}\right)_{\ell}\right).
\end{eqnarray*}
From Lemma \ref{lem0} it follows that for any $1 \le \ell \le n-1$,
$$
P\left(B_n| B_1\cap B_2 \cap\dots \cap B_{n-1}\right) <
P\left(B_n| \left(B_1\cap B_2 \cap\dots \cap B_{n-1}\right)_{\ell} \right).
$$
This implies
$
P(B_n)>P\left(B_n|B_1\cap \dots \cap B_{n-1}\right)\sum_{\ell=0}^{n-1} P\left(\left(B_1\cap \dots \cap B_{n-1}\right)_{\ell}\right).
$
The statement (\ref{c06}) follows since
the last sum  is equal to $1$.
\end{proof}
The proof of the theorem  \ref{main} is a verification of the conditions of the Borel-Cantelli lemma below.

\begin{lemma}{(\cite{Fri}, p.79)}
 \label{borel}
 Let random events $C_1, C_2, \dots$ be defined on one probability space
that satisfy the correlation condition:
for any natural $u$ and $v$, such that $u\not=v$,
random events $C_{u}$ and $C_{v}$ are negative correlated or uncorrelated.
If $\sum_{n=1}^{\infty} P(C_n)=\infty$, then $P\{C_n~\mbox{i.m.}\}=1.$\footnote{Through $\{C_n~\mbox{i.m.}\}$ denotes an event consisting in the fact that infinitely many events will occur from $C_1, C_2,...$ \cite{Shi}.}
\end{lemma}

\bigskip

\begin{lemma} \label{lem3} Let assumptions of Theorem \ref{main} hold. Then random events $B_1, B_2, \dots$
satisfy the correlation condition, i.e.,
\begin{equation} \label{lem3-0a}
P\left( B_{u} \cap B_{v} \right) \le P\left( B_{u}\right) P\left( B_{v}\right).
\end{equation}
\end{lemma}
\begin{proof}
 If $\min\{u,v\}=1$, then the supposition  (\ref{a04}) leads to $0<\#\left( {\tilde E}_1 \cap E_1\right)\le \#\left( {\tilde E}_1\right)$,
see, the proof of lemma \ref{lem0}. Thus from (\ref{c00}) follows that $P(B_1)>0$. If $\min\{u,v\}>1$, then combining (\ref{c03-0a}) and (\ref{c06}), we obtain $P\left( B_{\min\{u,v\}}\right)> p_{\min\{u,v\}}(G_1)>0$. It is proved that  $P\left( B_{\min\{u,v\}}\right)>0$ for any natural numbers $u$ and $v$.
Appplying definition of total probability one can rewrite the correlation condition as follows:
\begin{equation} \label{lem3-0a11}
P\left( B_{\max\{u,v\}}\right)-P\left( B_{\max\{u,v\}}| B_{\min\{u,v\}}\right) \ge 0.
\end{equation}
It is enough to derive that the inequality (\ref{lem3-0a11}) is fulfilled for
$\max\{u,v\}=n+k$  and $\min\{u,v\}=n$, where
$n$ and $k$  are any natural numbers.

If $k=1$, then applying the formula of total probability with (\ref{c03-0}), we get
$$
P(B_{n+1})-P(B_{n+1}|B_n)=\left(1-p_1(G_n)\right)\left(p_1(G_n)-p_2(G_n) \right) \ge 0.
$$
Here, the equality is valid for
$\#\left( {\tilde E}_n \cap E_n\right)= \#\left( {\tilde E}_n\right)$.

Let now $k \ge 2$. We state that
\begin{eqnarray}
  P\left( B_{n+k}\right) &=& J
+ \sum_{\ell=0}^{k-1} p_{n+k-\ell-1}(G_n)
P\left(B_n^c \cap \left(B_{n+1}\cap \dots \cap B_{n+k-1}\right)_{\ell} \right), \label{lem3-00} \\
P\left( B_{n+k}|B_n\right) &=&   J+ \sum_{\ell=0}^{k-1} p_{n+k-\ell}(G_n)
P\left(B_n^c \cap \left(B_{n+1}\cap \dots \cap B_{n+k-1}\right)_{\ell} \right), \label{lem3-01}
\end{eqnarray}
where
$$
J=\sum_{\ell=0}^{k-1} p_{n+k-\ell}(G_n)
P\left(B_n \cap \left(B_{n+1}\cap \dots \cap B_{n+k-1}\right)_{\ell} \right).
$$
From (\ref{lem3-00}), (\ref{lem3-01}) it follows immediately that the difference  $P\left( B_{n+k}\right)-P\left( B_{n+k}|B_n\right)$ is equal to
\begin{eqnarray*}
     && \sum_{\ell=0}^{k-1} \left\{p_{n+k-\ell-1}(G_n)-p_{n+k-\ell}(G_n)\right\}
P\left(B_n^c \cap \left(B_{n+1}\cap \dots \cap B_{n+k-1}\right)_{\ell} \right) \\
  &&=\sum_{\ell=0}^{k-1} \left\{p_{n+k-\ell-1}(G_n)-p_{n+k-\ell}(G_n)\right\}
\left\{1-p_{n-\ell}(G_n) \right\} P
\left(\left(B_{n+1}\cap \dots \cap B_{n+k-1}\right)_{\ell} \right).
\end{eqnarray*}
By (\ref{c03-0a}), (\ref{c03-0}) it follows that the latter sum is positive.

Let us prove (\ref{lem3-00}).
Applying the formula of total probability we obtain
\begin{eqnarray*}
  P\left(B_{n+k}\right) &=& P\left(B_{n+k} | \left(B_n\cap \dots \cap B_{n+k-1}\right)_{0} \right)
P\left(\left(B_n\cap \dots \cap B_{n+k-1}\right)_{0} \right) \\
&&+ \sum_{\ell=1}^{k-1} P\left(B_{n+k} | \left(B_n\cap \dots \cap B_{n+k-1}\right)_{\ell} \right)
P\left(\left(B_n\cap \dots \cap B_{n+k-1}\right)_{\ell} \right)\\
&&+ P\left(B_{n+k} | \left(B_n\cap \dots \cap B_{n+k-1}\right)_{k} \right)
P\left(\left(B_n\cap \dots \cap B_{n+k-1}\right)_{k} \right)
\end{eqnarray*}
Using the identity
\begin{eqnarray*}
\left(B_n\cap \dots \cap B_{n+k-1}\right)_{\ell}
&=&\left\{ B_n \cap \left(B_{n+1}\cap \dots \cap B_{n+k-1}\right)_{\ell}\right\} \cup
\\
&&\cup
\left\{ B_n^c \cap \left(B_{n+1}\cap \dots \cap B_{n+k-1}\right)_{\ell-1}\right\},
\end{eqnarray*}
we have
\begin{eqnarray*}
  P\left(B_{n+k}\right) &=& \bigg\{P\left(B_{n+k} | B_n\cap \dots \cap B_{n+k-1} \right)
P\left(B_n\cap \dots \cap B_{n+k-1} \right) \\
&&+ \sum_{\ell=1}^{k-1} P\left(B_{n+k} | \left(B_n \cap \dots \cap B_{n+k-1}\right)_{\ell} \right)
P\left(B_n \cap \left(B_{n+1}\dots \cap B_{n+k-1}\right)_{\ell}  \right) \bigg\} \\
&&+\sum_{\ell=1}^{k-1} P\left(B_{n+k} | \left(B_n \cap \dots \cap B_{n+k-1}\right)_{\ell} \right)
P\left(B_n^c \cap \left(B_{n+1} \cap \dots \cap B_{n+k-1}\right)_{\ell-1}  \right)  \\
&&+ P\left(B_{n+k} | B_n^c \cap \dots \cap B_{n+k-1}^c \right)
P\left(B_n^c\cap \dots \cap B_{n+k-1}^c \right).
\end{eqnarray*}
Applying (\ref{c03-n1}), we get that the sum in curly brackets is equal to $J$. By replacing the summation variable with  $s=\ell-1$, we obtain
\begin{eqnarray*}
&&\sum_{\ell=1}^{k-1} P\left(B_{n+k} | \left(B_n \cap\dots \cap B_{n+k-1}\right)_{\ell} \right)
P\left(B_n^c \cap \left(B_{n+1}\dots \cap B_{n+k-1}\right)_{\ell-1}  \right) \\
&&=\sum_{s=0}^{k-2} P\left(B_{n+k} | \left(B_n\cap \dots \cap B_{n+k-1}\right)_{s+1} \right)
P\left(B_n^c \cap \left(B_{n+1}\dots \cap B_{n+k-1}\right)_{s}  \right) \\
&&=\sum_{s=0}^{k-1} P\left(B_{n+k} | \left(B_n\cap \dots \cap B_{n+k-1}\right)_{s+1} \right)
P\left(B_n^c \cap \left(B_{n+1}\dots \cap B_{n+k-1}\right)_{s}  \right) \\
&&-P\left(B_{n+k} | \left(B_n\cap \dots \cap B_{n+k-1}\right)_{k} \right)
P\left(B_n^c \cap \left(B_{n+1}\dots \cap B_{n+k-1}\right)_{k-1}  \right).
\end{eqnarray*}
Using (\ref{c03-n1}) once more, we will get the proof (\ref{lem3-00}).

The identity (\ref{lem3-01}) can be checked in a similar way. To this end, it is enough to apply
the formula of total probability and the identity
\begin{eqnarray*}
\left(B_{n+1}\cap \dots \cap B_{n+k-1}\right)_{\ell} &=&
\left\{ B_n \cap \left(B_{n+1}\cap \dots \cap B_{n+k-1}\right)_{\ell}\right\} \cup
\\
&&
\cup \left\{ B_n^c \cap \left(B_{n+1}\cap \dots \cap B_{n+k-1}\right)_{\ell}\right\}.
\end{eqnarray*}
\end{proof}

\bigskip

\subsection{Proof of Theorem \ref{main}}
\begin{proof}
The difference of the total number of triangles in graphs $G_n$ and  $G_1$ as well as the difference of the number of nodes involved in triangles  can be expressed through indicators of random events $B_j$ by
\begin{eqnarray}
\label{c01}
  \Delta_n-\Delta_1 &=& \sum_{j=1}^{n-1} I\left\{ B_j\right\},\\
  \label{c02}
\#{\tilde V}_n -\#{\tilde V}_1&=& \sum_{j=1}^{n-1} I\left\{ B_j\right\}, \quad n \ge 2,
\end{eqnarray}
where $I\left\{ \cdot \right\}$ denotes the indicator of the event. By (\ref{c01}), (\ref{c02}) it follows that the
counting of the total number of triangles $\Delta_n$ is a linear function of the cardinality of the set ${\tilde V}_n$:
\begin{equation} \label{c02-1}
\Delta_n= \#{\tilde V}_n +\left(\Delta_1-\#{\tilde V}_1\right), \quad n \ge 2.
\end{equation}
In view of (\ref{c01}), the statement (\ref{a06}) follows from
$$
\sum_{j=1}^{n-1} I\left\{ B(j)\right\} \to \infty \quad {\rm a.s.} \qquad \mbox{as}\qquad\ n\to \infty,
$$
that equivalently implies that
\begin{eqnarray*} 
&& P\left(B_n \ {\rm occurs 
\ infinitely \ many \ times} \ n\right)=1.
\end{eqnarray*}
Let us prove that the sequence  $B_1, B_2, \dots$
satisfies the conditions of Borell-Cantelli lemma (see, lemma \ref{borel}). One of them is verified in the lemma \ref{lem3}. It is enough to prove that
\begin{equation}
    \label{c10}
\sum_{n=1}^\infty P\left(B_n\right)= \infty.
\end{equation}
Due to lemmas  \ref{lem0}, \ref{lem2},
$$
\sum_{n=1}^\infty P\left(B_n\right) >  \sum_{n=3}^\infty p_n(G_1).
$$
Note that  by (\ref{c03-n1}), it follows
\begin{equation}
    \label{c11-a}
p_n(G_1) > \frac{\#\left({\tilde E}_1 \cap E_1 \right)+2(n-1)}{n \# {\tilde E}_1+(1/2)(n-1)(n-2)}.
\end{equation}
Using inequality  $\#\left({\tilde E}_1 \cap E_1 \right) \le \#{\tilde E}_1$, one can prove that the right-hand side of   (\ref{c11-a}) exceeds
$\#\left({\tilde E}_1 \cap E_1 \right)/\left(\left( 3 \# {\tilde E}_1 \right) (n-2)\right)$ for $n\ge 3$. Hence, it holds
$$
\sum_{n=1}^\infty P\left(B_n\right) >  \frac{\#\left({\tilde E}_1 \cap E_1 \right)}{ 3 \# {\tilde E}_1} \lim_{N \to \infty} H_N,
$$
where $H_N=\sum_{n=1}^{N} 1/n$.
The relation (\ref{c10}) follows by applying
the inequality $H_N>\ln(N)$ and the relation $\lim_{N\to \infty}\ln(N)=\infty$.
\end{proof}

It remains to prove corollaries  \ref{cor0}, \ref{cor}. The statement of the corollary \ref{cor0} follows from (\ref{a06}) and (\ref{c02-1}).

To prove the corollary \ref{cor}, we note that from
(\ref{c01}) the inequality follows
$$
E\left(\Delta_n\right)-\Delta_1=\sum_{j=1}^{n-1} P(B_j) > \sum_{j=3}^{n+1} p_{j}(G_1), \quad n \ge 2.
$$
It follows from the proof of the relation (\ref{c10})  that the last sum is greater than
$\left( \#\left({\tilde E}_1 \cap E_1 \right)/\left( 3 \# {\tilde E}_1 \right)\right) \ln(n-1)$. This completes the proof of the inequality  (\ref{a07}).

\section{Monte-Carlo simulation} \label{sec4}

Let us consider modeling the sequence $\Delta_n$, $n\ge 1$.
To investigate the growth of the sample mean $\Delta_n$
modeling a sequence of graphs $G_n$, $n\ge 1$ is
not required. It is enough to simulate a sequence of sets ${\tilde E}_n$, $n\ge 1$.

Let us interpret pairs of nodes from ${\tilde E}_n$ in terms of  a generalized  P\'{o}lya-Eggenberger urn model. We use the fact that
for each $n\ge 1$, the set
${\tilde E}_n$ is divided into two subsets: ${\tilde E}_n \cap E_n$ and ${\tilde E}_n \setminus \left( {\tilde E}_n \cap E_n\right)$, see section \ref{sec4}.
Each pair of nodes from ${\tilde E}_n \cap E_n$ is interpreted as a white ball, while a pair of nodes belonging to the set ${\tilde E}_n \setminus \left( {\tilde E}_n \cap E_n\right)$
is interpreted as a black ball. Hence, an urn (or in other words, the set ${\tilde E}_n$) contains $\# \left({\tilde E}_n \cap E_n\right)$ white balls and
$\# \left({\tilde E}_n \setminus \left( {\tilde E}_n \cap E_n\right)\right)$ black ones.
\\
At every time step $n\ge 1$, we draw a ball  uniformly (with return) at random from the urn (see, proposition \ref{prop1}).  The color of the ball is inspected and  the urn is replenished according
to the following rule. If a white ball is drawn, then we put into the urn two white balls and $\left(1/2)\left( 1+\sqrt{1+8\# {\tilde E}_n }\right)-2\right)$ black balls.
 We recall that cardinalities $\# {\tilde V}_n$ and $\# {\tilde E}_n$
are related as $\# {\tilde V}_n\left(\# {\tilde V}_n-1\right)=2 \# {\tilde E}_n$ (see, (\ref{20})), and that is why $\# {\tilde V}_n+1$ ball is added to the urn.  If we have chosen a black ball, the content of the urn is not changed.

Note that the number of added black and/or white balls is fixed in classical P\'{o}lya-Eggenberger urn model (see, e.g., \cite{Fri}, pg. 437)
while the
number of added black balls depends on $\# {\tilde E}_n$ in our case. For more generalizations of the P\'{o}lya-Eggenberger urn model we refer to
\cite{Chen} and references therein.

Keeping in mind that  pairs of nodes from ${\tilde E}_n$ are uniformly distributed (see Prop. \ref{prop1}), we provide quite simple algorithm
to simulate the sequence $\Delta_n$, $n\ge 1$.

\begin{algorithm}\label{alg:1}
\ 

\begin{enumlist}[.]
  \item The initial  step. Using an initial graph $G_1$, calculate
    $\Delta_1$, $\#{\tilde V}_1$,  $\# \left({\tilde E}_1 \cap E_1\right)$ and
$\#\left({\tilde E}_1 \setminus \left( {\tilde E}_1 \cap E_1\right)\right)$.
   \item  The evolutionary step. For any $n \ge 1$,  simulate a variate of random variable $\xi_n$, which has the uniform discrete distribution at the interval of natural numbers $[1, \dots, \# {\tilde E}_n ]$.
If $\xi_n>\# \left({\tilde E}_n \cap E_n\right)$, then
\begin{eqnarray*}
  \# \left({\tilde E}_{n+1} \cap E_{n+1}\right)&=&\# \left({\tilde E}_n \cap E_n\right), \\
\# \left({\tilde E}_{n+1} \setminus \left( {\tilde E}_{n+1} \cap E_{n+1}\right)\right)&=&
\# \left({\tilde E}_{n} \setminus \left( {\tilde E}_{n} \cap E_{n}\right)\right).
\end{eqnarray*}
If $\xi_n \le \# \left({\tilde E}_n \cap E_n\right)$, then
\begin{eqnarray*}
  \# \left({\tilde E}_{n+1} \cap E_{n+1}\right)&=&\# \left({\tilde E}_n \cap E_n\right)+2, \\
\# \left({\tilde E}_{n+1} \setminus \left( {\tilde E}_{n+1} \cap E_{n+1}\right)\right)&=&\# {\tilde E}_{n+1} - \# \left({\tilde E}_{n+1} \cap E_{n+1}\right),
\end{eqnarray*}
where
$$
\# {\tilde E}_{n+1} :=\# \left({\tilde E}_n \cap E_n\right)+\# \left({\tilde E}_{n} \setminus \left( {\tilde E}_{n} \cap E_{n}\right)\right)+
(1/2)\left(1+\sqrt{1+8\# {\tilde E}_{n+1}}\right).
$$
In both cases, we get
$$
\Delta_{n+1}=\Delta_1-\# {\tilde V}_1 +(1/2)\left(1+\sqrt{1+8\# {\tilde E}_{n+1}}\right).
$$
\end{enumlist}
\end{algorithm}
Let us recall that we need to store only four quantities during the realization of the proposed algorithm: $\Delta_1$, $\#{\tilde V}_1$, $\# \left({\tilde E}_n \cap E_n\right)$ and $\# \left({\tilde E}_{n} \setminus \left( {\tilde E}_{n} \cap E_{n}\right)\right)$.

For each  initial graph $G_1$  $100$ independent sequences $\Delta_1^{(j)}, \dots, \Delta_N^{(j)}$ of the length $N=10^6$ are  simulated.
The upper index $(j)$
denotes the number of sequence.
\\
The following complete graphs are selected as the initial ones: 1) triangle $G_{1,1}$; 2) complete graph $G_{1,2}$ with the number of vertices $17$; 3) complete graph $G_{1,3}$ with the number of vertices $51$.


  The sequence of sample means  of triangle counts  $\Delta_{n}^{(j)}$, $1\le n \le N$ is calculted by formula
$$
{\bar \Delta}_{n}(G_{1, \ell}) = \frac{1}{100} \sum_{j=1}^{100} \Delta_{n}^{(j)}, \quad \ell\in\{1,2,3\}
$$
for each initial graph $G_{1, \ell}$.
The plots $\left\{\left(n, {\bar \Delta}_{n}(G_{1, \ell})- \Delta_{1}(G_{1, \ell})\right), \ 1\le n \le N\right\}$, where $\ell\in\{1,2,3\}$ and $N=10^6$ are shown in Fig.\ref{fig2}.

\begin{figure}[ht!]
\begin{center}
\includegraphics[width=0.70\textwidth]{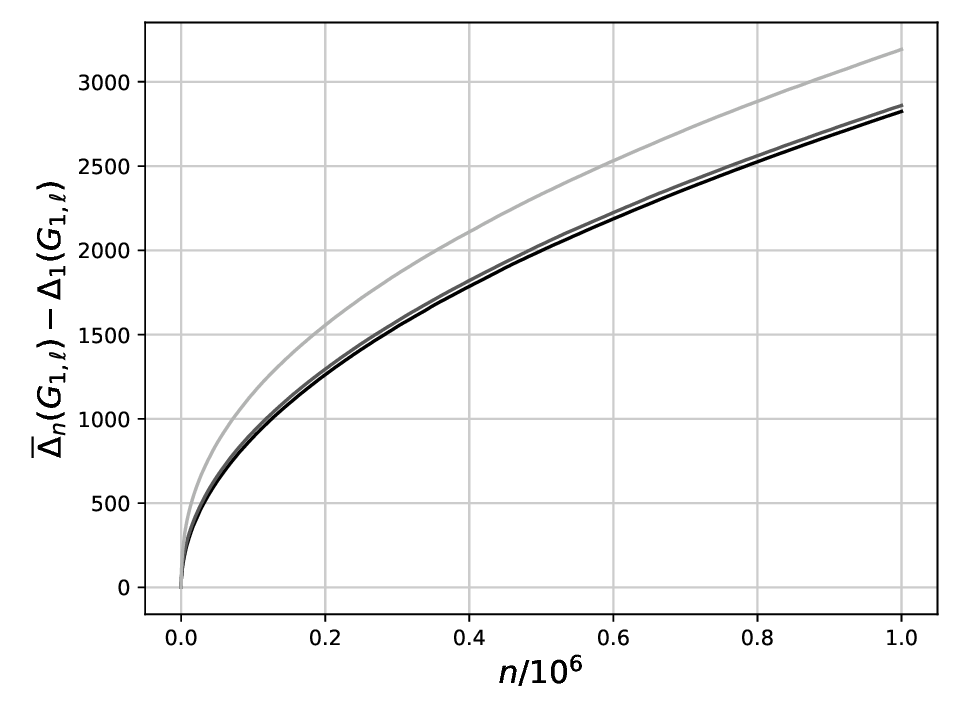}
\end{center}
\caption{Plots $\left\{\left(n, {\bar \Delta}_{n}(G_{1, \ell})-\Delta_{1}(G_{1, \ell})\right), \ 0\leq n \leq 10^6\right\}$ where plots with $\ell\in\{1,2,3\}$ are shown in black, grey and light grey collors, respectively.}\label{fig2}
\end{figure}

The least squares method is used to fit the discrete data
$$
\left\{ \left(n, {\bar \Delta}_{n}(G_{1, \ell})-\Delta_{1}(G_{1, \ell})\right), \  2 \cdot 10^5 \le n \le N\right\},
$$
$$
\left\{ \left(n, {\bar \Delta}_{n}(G_{1, \ell})-\Delta_{1}(G_{1, \ell})\right) , \  5 \cdot 10^5 \le n \le N\right\}
$$
by the function
$\varphi(n)=c_1+c_2 n^{c_3}$. Estimates of $c_1$, $c_2$ and $c_3$ are presented in Table \ref{tab1}.

\begin{table} [ht!]
\begin{center}
\begin{tabular}{|c|c|c|c|}
 \hline
 \multicolumn{4}{|c|}{The range of $n$: $[2 \cdot 10^5, 10^6]$} \\
 \hline
 Initial graph & ${\hat c}_1$ & ${\hat c}_2$ & ${\hat c}_3$ \\
 \hline
 $G_{0,1}$ & 1  & 2.827 & 0.499 \\
 $G_{0,2}$ & 680  & 3.179 & 0.492 \\
 $G_{0,3}$ & 20825  & 6.564 & 0.447 \\
\hline
\end{tabular}
\quad
\begin{tabular}{|c|c|c|c|}
 \hline
 \multicolumn{4}{|c|}{The range of $n$: $[5 \cdot 10^5, 10^6]$} \\
 \hline
 Initial graph & ${\hat c}_1$ & ${\hat c}_2$ & ${\hat c}_3$ \\
 \hline
 $G_{0,1}$ & 1  & 2.831 & 0.499 \\
 $G_{0,2}$ & 680  & 3.234 & 0.491 \\
 $G_{0,3}$ & 20825  & 6.133 & 0.452 \\
 \hline
\end{tabular}
\end{center}
\caption{ Estimates of $c_1$, $c_2$ and $c_3$.}\label{tab1}
\end{table}
The findings of our simulation are the following.
\begin{enumerate}
  \item Simulation results shown in Fig. \ref{fig2} are approximately similar to the theoretical ones. Indeed, $E\Delta_n$ and empirical mean ${\bar \Delta}_n$
  grow together with $n$, see corollary \ref{cor} and fig. \ref{fig2}.
  \item
  ${\bar \Delta}_n$ grows as $C \sqrt{n}$, where $2 \cdot 10^5 \le n \le 10^6$, see Tab. \ref{tab1}.
  This simulation result does not contradict to
  Corollary \ref{cor}. It should be noted that
  from this simulation result it does not follow that ${\bar \Delta}_n$ has the same (or close) growth rate as $n > 10^6$.
  \item When modeling the number of triangles $\Delta_n$ at $n >10^6$, difficulties arose due to the limited computer memory.
\item The plots of $\left(n, {\bar \Delta}_{n}(G_{1,1})-\Delta_{1}(G_{1, 1})\right)$ and $\left(n, {\bar \Delta}_{n}(G_{1,2})-\Delta_{1}(G_{1, 2})\right)$ are in the close correspondence, while the plot  $\left(n, {\bar\Delta}_{n}(G_{1,3})-\Delta_{1}(G_{1, 3})\right)$ has a shift along the vertical axis relative to the other two graphs at $0\le n\le 10^6$, which is most likely due to the size of the graph $G_{1,3}$.
It can be assumed that the influence of a relatively large initial graph should be eliminated with a large $n$.
  \end{enumerate}

\section{Gratitude}\label{sec7}
The authors express their gratitude to M.S. Ryzhov for the computer simulation.

\section{Conclusions}\label{sec5}

Motivated by \cite{Wan}, where rather wide  definition of the PA models
is given, we provide the new definition of the CA  model. By means of the latter
one can 
simulate the evolution of many random systems and networks where newly appending nodes are drawn likely not towards hubs,
but towards densely connected groups.

The properties of our CA model in which new nodes connect to existing nodes with equal probabilities are investigated. The new CA model can be qualified
as a limit (as the parameter of the attachment function $\alpha \downarrow 0$) of the CA model proposed in \cite{Bag}. We proved that the total number of triangles $\Delta_n$
tends to infinity almost surely as evolution step $n\to\infty$ without additional assumptions  on the choice of the initial graph from which the CA begins. To prove this result, it is proved that there is a negative correlation or uncorrelation of random events, consisting in the fact that pairs of existing nodes selected by new nodes at different time moments are connected by an edge, or in other words, that a new triangle is formed as a result of appending a new node. It is also shown that as the evolution step increases, the probability of obtaining a new triangle decreases, which can lead to stabilization of the growth in the number of triangles. The corollary \ref{cor} shows that the growth rate of the average number of triangles $E\Delta_n$ is higher than the logarithmic one with increasing number of evolution steps. These results are of independent interest.

Our simulation results do not contradict to the theoretical results. Specification of the upper and lower bounds for $E\Delta_n$ is
the purpose of our further research.






\end{document}